\documentclass[11pt]{amsart}

\usepackage{ amstext, amsmath, amssymb, graphicx, array, epsfig, xcolor,epstopdf}
\usepackage{eucal,amsfonts,mathrsfs,amscd,bm,tcolorbox}
\usepackage{booktabs, caption, subcaption}
\usepackage{multirow}

\newtheorem{theorem}{Theorem}[section]
\newtheorem{lemma}[theorem]{Lemma}

\theoremstyle{definition}

\theoremstyle{remark}
\newtheorem{remark}[theorem]{Remark}

\newcommand{\bfzero}{\boldsymbol 0}
\newcommand{\bfbeta}{\boldsymbol \beta}

\newcommand{\bff}{\boldsymbol f}

\newcommand{\bfn}{\boldsymbol n}

\newcommand{\bfx}{\boldsymbol x}

\newcommand{\bfe}{\boldsymbol e}

\newcommand{\bfu}{\boldsymbol u}
\newcommand{\bfv}{\boldsymbol v}
\newcommand{\bfw}{\boldsymbol w}

\newcommand{\bfV}{\boldsymbol V}
\newcommand{\bfQ}{\boldsymbol Q}
\newcommand{\bfW}{\boldsymbol W}

\newcommand{\bfpsi}{\boldsymbol \psi}

\newcommand{\mcT}{\mathcal{T}}

\numberwithin{equation}{section}

\newcommand{\jump}[1]{[\![#1]\!]}

\newcommand{\bfeta}{{\boldsymbol \eta}}
\newcommand{\bfm}{{\boldsymbol m}}
\newcommand{\calB}{\mathcal{B}}
\newcommand{\bfD}{{\boldsymbol D}_h}
\newcommand\calT{\mathcal{T}}
\newcommand\calTt{\mathcal{\tilde T}}
\newcommand\calF{\mathcal{F}}
\newcommand{\calV}{\mathcal{V}}
\newcommand{\calM}{\mathcal{M}}
\newcommand{\calE}{\mathcal{E}}
\newcommand{\bcalV}{\boldsymbol{\mathcal V}}
\newcommand\bZ{\boldsymbol{Z}}
\newcommand\bz{\boldsymbol{z}}
\newcommand\bfalpha{\boldsymbol{\alpha}}
\newcommand{\dive}{{\ensuremath\mathop{\mathrm{div}\,}}}
\newcommand{\curl}{{\ensuremath\mathop{\mathrm{curl}\,}}}

\newcommand{\pol}{\mathbb{P}}

\newcommand{\bl}{\bigl\langle}
\newcommand{\br}{\bigr\rangle}

\newcommand{\vertiii}[1]{{\left\vert\kern-0.25ex\left\vert\kern-0.25ex\left\vert #1 
    \right\vert\kern-0.25ex\right\vert\kern-0.25ex\right\vert}}

\usepackage{graphicx}

\begin{document}

\title[CIP for divergence-free FEM]{Continuous Interior Penalty stabilization for divergence-free finite element methods}

\author[G.R.~Barrenechea]{Gabriel R. Barrenechea\textsuperscript{\textdagger}}
\address{\textsuperscript{\textdagger}Department of Mathematics and
    Statistics, University of Strathclyde, 26 Richmond Street,
    Glasgow, G1 1XH United Kingdom}
\email{gabriel.barrenechea@strath.ac.uk}

\author[E.~Burman]{Erik Burman\textsuperscript{\textdaggerdbl}}
\address{\textsuperscript{\textdaggerdbl}Department of Mathematics, University College London, London, UK-WC1E  6BT, United Kingdom}
\email{e.burman@ucl.ac.uk}

\author[E.~C\'aceres]{Ernesto C\'aceres\textsuperscript{\textsection}}
\address{\textsuperscript{\textsection}Department of Mathematical Sciences,
	Worcester Polytechnic Institute,
100 Institute Road,
Worcester, MA, 01609}
\email{ecaceres@wpi.edu}

\author[J.~Guzm\'an]{Johnny Guzm\'an\textsuperscript{\textparagraph}}
\address{\textsuperscript{\textparagraph}Division of Applied Mathematics,
Brown University,
Box F,
182 George Street,
Providence, RI 02912}
\email{johnny\_guzman@brown.edu}

\maketitle

\begin{abstract}
In this paper we propose, analyze, and test numerically a pressure-robust stabilized finite element for a linearized problem in incompressible fluid mechanics, namely, the steady Oseen equation with low viscosity. Stabilization terms are defined by jumps of different combinations of derivatives for the convective term over the element faces of the 
triangulation of the domain. With the help of these stabilizing terms, and the fact the finite element space is assumed to provide a point-wise
divergence-free velocity, an $\mathcal O\big(h^{k+\frac12}\big)$ error estimate in the $L^2$-norm is proved for the method (in the convection-dominated regime), and optimal order
estimates in the remaining norms of the error. 
Numerical results supporting the theoretical findings are provided.
\end{abstract}


\noindent AMS class: 65N30;  65N12;  76D07

\section{Introduction}
The design and analysis of accurate and robust computational methods for high Reynolds flow remains a
challenging topic in computational mathematics. An important difficulty is to find a dissipative mechanism that is sufficiently strong to eliminate underesolved high frequency content, and prevent unphysical transport velocities, while having minimal effect on the resolved scales.
One approach that has proven successful is to use stabilized finite element methods to counter unphysical fine scale accumulation of energy. Such stabilized finite element methods using continuous approximation spaces include the Streamline Upwind Petrov Galerkin method \cite{BH82}, the Subgrid Viscosity methods \cite{Guer99}, the Orthogonal Subscale method \cite{Cod08}, the Local Projection method \cite{BB04,BB06}
(or its residual version \cite{BV10,ABPV12}) and the Continuous Interior Penalty (CIP) method \cite{DD76,BFH06}. In the case of discontinuous approximation spaces the use of penalty on the solution jump (or upwind flux) provides a similar stabilizing mechanism \cite{JP86}.
From the theoretical point of view one may prove on linear scalar model problems that in the high P\'eclet regime such stabilized methods satisfy an $O(h^{k+\frac12})$ bound (here $k$ denotes the polynomial order of the approximation space) for the error in the $L^2$-norm in parts where the solution is smooth \cite{JNP84, JSW87, Guz06, BGL09,dFGJN19}.
This means that the method is close to optimal for smooth solutions and that rough features in the solution cannot spuriously influence the approximation in the smooth zone. It is also known that for scalar model problems the discontinuous Galerkin method using penalty on the solution jumps and the stabilized methods have superior dispersion properties compared to the standard Galerkin method \cite{MMGPS18, MCSBS22}.
In the case of the equations of incompressible flow only global estimates in the $L^2$-norm of $O(h^{k+\frac12})$ for smooth solutions have been proved in the case of equal order interpolation  \cite{JS86, HS90, BF07, GJN21}, or by adding bubble functions to the finite element space  \cite{MT15}. 
Nonwithstanding the lack of theoretical foundations in the nonsmooth case, there is ample computational evidence that stabilized methods indeed performs well also in the case of turbulent flows \cite{JWH00,GWR04,HJ07,BCCHRS07,MCSBS22}. In view of this it does not seem to far fetched to use the $O(h^{k+\frac12})$ estimate for smooth problems as a theoretical proxy for good performance for general high Reynolds flows.

Recently there has been a surge of interest in methods that allow for a pointwise divergence free approximation \cite{JLMNR17}. This means that mass conservation is satisfied exactly. It has been argued that such methods have advantages in the situation where the viscosity becomes small, however in the discussion most often the convective term has been omitted and a degenerate Stokes' problem considered instead. This is not physically realistic, since as the Reynolds number increases the convective term dominates.
Nevertheless an appealing feature of such methods is that the error bounds for the velocity are independent of the Sobolev norm on the pressure, which in general scales as $\mu^{-1}$, where $\mu$ is the viscosity. The methods are therefore said to be pressure robust. For the Oseen equation it was noticed in \cite{BL08} that the application of standard stabilized finite element methods to such pressure robust discretizations was not straightforward. In particular the fact that the approximate solution was constrained to be in the divergence free space blocked the derivation of the $O(h^{k+\frac12})$ estimate if standard stabilization methods such as the CIP method were applied. Only recently was $O(h^{k+\frac12})$
$L^2$-error estimates proven for pressure robust discretizations of incompressible flow. The first instance was for $H(\mathrm{div})$ conforming approximation spaces \cite{BBG20} where a linear inviscid problem in incompressible flow was considered, showing that earlier results for the incompressible Euler equations \cite{GSS17} indeed also enjoys the $O(h^{k+\frac12})$ estimate (see also \cite{HH21} for the time-dependent Navier-Stokes equations). In this case the upwind flux provides the stabilizing mechanism. Computational evidence of the good performance of methods using $H(\mathrm{div})$ conforming approximation for high Reynolds number flows have already been reported in \cite{SL17,SL18,SLLL18,LLS18,LLS19,GSS17,NC18}. A second approach using $H^1$-conforming approximation followed, using Galerkin Least squares stabilization in the vorticity equation \cite{BarrBurmanGuzman}, where the method was presented and analyzed for the Oseen equation. It was later shown in \cite{BHL22} that in the lowest order case of piecewise affine velocity approximation this type of stabilization could be related to the classical Smagorinsky turbulence model.  The stabilizing term in \cite{BarrBurmanGuzman} consisted in a residual term in the bulk, supplemented with jump terms on the skeleton. When the time-dependent case is considered, the fact that the bulk term is residual can constrain the choice for time discretization severely.

 A parallel development in the isogeometric community has been to introduce versions of the CIP methods where jumps of higher order derivatives are penalized over the skeleton \cite{HVABR18,HVQARB19}. In the context of the modelling of high Reynolds flows this connects with earlier ideas related to time relaxation methods \cite{BBH11}, where stability was obtained by penalizing the $L^2$-distance to a smoother approximation of the flow, hence penalizing the jumps in all derivatives.
 Apart from the obvious connection to the CIP method pointed out in \cite{BBH11}, when finite elements are used, so far there has been no numerical analysis of these high order skeleton stabilized methods. Recently such a skeleton stabilization was proposed also for divergence-conforming B-spline approximations of incompressible flows \cite{tong2022skeletonstabilized}, also in this case without theoretical justification. Interestingly, the method presented in \cite{tong2022skeletonstabilized} is closely related to the classical CIP method when $C^0$ splines (i.e., classical conforming finite element spaces) are used.

The objective of the present paper is to revisit the stabilization of \cite{BarrBurmanGuzman} and show that using the ideas of the analysis of the CIP stabilization the bulk term can be eliminated by replacing it with two skeleton based stabilization terms. This 
%disconnects the stabilization from the time derivative in the time dependent case, allowing for general time stepping schemes and, maybe more importantly, 
sheds some light on the precise quantities that need to be controlled on the faces of the mesh skeleton in order to prove improved error estimates for the velocity in a pressure-robust way. 
Our notations and analysis follow from \cite{BarrBurmanGuzman}, as both methods add stability by controlling the curl of the convective term, but significant
differences appear in the definition of the method, and its analysis.

\subsection{Outline of the paper}
The paper is organized as follows: the introduction is completed by one short section  regarding the motivation and background for the new proposed method, and 
some preliminary results about vector potentials for divergence-free functions and their regularity. In Section~\ref{sec:fem:spaces}, we introduce the 
%original mesh and its barycentric refinement, and the respective 
finite element method used in this work. Additionally, we introduce abstract properties that the discrete spaces must satisfy to ensure well-posedness of the discrete problem and optimal error estimates for the velocity with respect to a norm that is induced by the discrete bilinear form. We finish Section~\ref{sec:fem:spaces} with the definition of the stabilized finite element method. In Section~\ref{sec:numerical:analysis}, we prove an error estimate for the velocity with respect to the aforementioned norm. In Section~\ref{sec:CTFEM} we make a precise choice for the finite element spaces, namely, we focus on the Scott-Vogelius element and prove the hypotheses
introduced earlier for a barycentrically refined mesh linking it to the lowest order Clough-Tocher space. We finish the paper in Section~\ref{sec:numerics} by showing numerical examples showing both convergence of the finite element solution in the case of a smooth solution, and in the stabilization properties of
the present method in the presence of sharp layers showing, in particular, that the use of CIP alone is not sufficient to stabilize layers, and justifying
the addition of the extra terms to the method.

\subsection{Background and preliminary results.}\label{sec:1.2} 

In this work we adopt standard notation for Sobolev and Lebesgue spaces, aligned, e.g., with \cite{EG21-I}. More precisely, for $D\subseteq\mathbb{R}^d, 1\le d\le 3$ will be denoted
by $(\cdot,\cdot)_D^{}$ (if $D\subseteq\mathbb{R}^{d-1}$ the inner product will be denoted
by $\langle\cdot,\cdot\rangle_D^{}$). The norm in $L^2(D)$ will be denoted by $\|\cdot\|_{0,D}^{}$, and for $m\ge 0, p\in [1,\infty]$ the norm (seminorm) in $W^{m,p}(D)$
will be denoted by $\|\cdot\|_{m,p,D}^{}$ ($|\cdot|_{m,p,D}^{}$).
 If $p=2$ as usual we denote $H^m(D)=W^{m,2}(D)$ with norm (seminorm) denoted by $\|\cdot\|_{m,D}^{}$ ($|\cdot|_{m,D}^{}$).
No distinction will be made between inner products and norms for scalar and vector (or tensor)-valued functions.

We consider a linearized version of the stationary Navier--Stokes equations, namely Oseen's problem, posed on a polyhedral, bounded, connected and Lipschitz domain $\Omega$:
\begin{align}\label{eq:oseen}
- \mu \Delta \bfu + (\bfbeta \cdot \nabla ) \bfu  + \sigma \bfu + \nabla p & = \bff \qquad \textrm{in}\;\Omega\,,\nonumber\\
\dive \bfu & = 0\qquad \textrm{in}\;\Omega\,,\\
\bfu & = \bfzero\qquad\textrm{on}\; \partial \Omega\,. \nonumber
\end{align}
Here, $\mu>0$ denotes the viscosity coefficient, $\sigma > 0$ is the reactive coefficient, and the convective term $\bfbeta$ is assumed to
%belong to $W^{3,\infty}(\Omega)^d$ and to 
satisfy $\dive \bfbeta = 0$.
The standard weak formulation of \eqref{eq:oseen} is: find $(\bfu,p)\in \bfV\times \bfQ:=H^1_0(\Omega)^d\times L^2_0(\Omega)$ such that
\begin{equation}\label{eq:weak:Oseen}
\hspace*{-0.4cm}
\left\{
\begin{array}{rl}
a(\bfu,\bfv)+b(p,\bfv) &= (\bff, \bfv )_\Omega^{}\,,\\
b(q,\bfu) &=  0\,,
\end{array}
\right.
\end{equation}
for all $(\bfv,q)\in \bfV\times \bfQ$, where the bilinear forms are defined as
\begin{align}
a(\bfu,\bfv) :=&  \,(\sigma \bfu + (\bfbeta \cdot \nabla) \bfu,\bfv)_\Omega^{} + \mu(\nabla \bfu,\nabla \bfv)_\Omega^{}\,,  \label{def-a}\\
b(q,\bfv) :=&  \,-(q, \dive\bfv)_\Omega^{}\,. \label{def-b}
\end{align}
Defining the space
\begin{equation}\label{eq:div-free-space}
\bcalV(\Omega) = \{\bfv\in H^1_0(\Omega)^d : \dive\bfv =0\;\; \textrm{in}\;\Omega\}\,,
\end{equation}
then we can also write the following pressure-independent weak formulation of Oseen's problem: Find $\bfu \in \bcalV(\Omega)$ such that for all $\bfv \in \bcalV(\Omega)$ the following holds
\begin{equation} \label{eq:weak:Oseen-reduced}
  \mu (\nabla \bfu, \nabla \bfv)_\Omega^{} + ((\bfbeta \cdot \nabla ) \bfu, \bfv)_\Omega^{}
    +  \sigma (\bfu, \bfv)_\Omega^{}  = (\bff, \bfv )_\Omega^{}.
\end{equation}
While \eqref{eq:weak:Oseen} is the formulation that is used computationally, \eqref{eq:weak:Oseen-reduced} is very helpful to carry out the analysis in some instances, as both formulations
are equivalent. The mixed formulation \eqref{eq:weak:Oseen} is well-posed in
$H^1_0(\Omega)^d \cap \bcalV(\Omega) \times L^2_0(\Omega)$ by Lax--Milgram's lemma
and Brezzi's theorem for all $\mu>0$ (for details in these last two points, see, e.g., \cite{GR86}).

\medskip

Assuming that the domain $\Omega$ is contractible and Lipschitz, 
it is known   that we can associate a potential to every divergence-free function in $\Omega$.  The space that capture the kernel of the divergence operator is 
\begin{equation*}
\begin{array}{lr}
\bZ:=\{ \bz: \text{ components of } \bz \text{ belong to } H^1(\Omega),   \curl \bz \in H_0^1(\Omega)^d\}.
\end{array}
\end{equation*}
When $d=3$, $\bz$ is a vector-valued function whereas when $d=2$,  $\bz$ is a scalar function for. Also, here when $d=2$ the $\curl$ operator maps a scalar to a vector: $\curl v=(\partial_{x_2} v, -\partial_{x_1} v)^T$. We stress the fact that the boldface notation is being used for $d=2$ and $d=3$, even when $\bz$ is a scalar for $d=2$. We have additional regularity for $\bz$ if $\bfu$ is more regular.  In particular, since $\Omega$ is contractible and Lipschitz,  if  $\bfu \in  H^r(\Omega)^d\cap H^1_0(\Omega)^d$ with $\dive \bfu=0$  then there exists $\bz \in  \bZ$  with components in $H^{r+1}(\Omega)$  \cite{Costabel} such that $\curl \bz=\bfu$ with the following bound $\|\bz\|_{H^{r+1}(\Omega)} \le C \|\bfu\|_{H^r(\Omega)}$. 

Finally, below when $d=2$ we will also need the scalar $\curl$ operator which maps a vector to a scalar: $\curl v=\partial_{x_2} v_1-\partial_{x_1} v_2$. Hence, when $d=2$ the $\curl$ operator must be interpreted appropriately, depending if it is acting on a scalar or vector-valued function.  

\section{The stabilized finite element method}
\label{sec:fem:spaces}

\subsection{Finite element spaces}\label{sec:2.1}
Let $\{\calT_h^{} \}_{h>0}^{}$ be a family of shape-regular triangulations of
$\Omega$. The elements of $\calT_h^{}$ will be denoted by $T$, with corresponding diameter $h_T^{}:=\textrm{diam}(T)$ and maximal mesh width $h=\max\{h_T^{}:T\in\calT_h^{}\}$.
Given $T\in \calT_h^{}$, we denote by $\calF_T^{}$ the set of its facets. The collection of facets from the triangulation $\calT_h^{}$ is denoted by $\calF$, and $\calF^i$ will denote the interior facets . For $F\in \calF$ we define $h_F^{}=\textrm{diam}(F)$, and denote by $|F|$ the $(d-1)$-dimensional measure of $F$. 
%The $L^2(F)$-inner product is denoted by $\bl\cdot,\cdot\br_F^{}$.
For a vector valued function $\bfv$ we define the tangential jumps across $F= T_1  \cap T_2$ with $T_1, T_2 \in \calT_h$  as
\begin{equation*}
\jump{\bfv \times \bfn}|_F := \bfv_1 \times \bfn_1 + \bfv_2 \times \bfn_2,
\end{equation*}
where $\bfv_i=\bfv|_{K_i}$ and $\bfn_i$ is the unit normal pointing outwards $T_i$. If $F$ is a boundary face then we define
\begin{equation*}
\jump{\bfv \times \bfn}|_F^{} := \bfv \times \bfn.
\end{equation*}
The following broken inner products for regular enough functions $u,v$ (here $\mathcal{G} \subset \calF$) will be needed
\begin{equation}\label{broken-pi}
(v,w)_h^{}:=\sum_{T\in\calT_h^{}}(v,w)_T^{}, \quad  \bl v,w \br_{\mathcal{G}}:=\sum_{F \in \mathcal{G}} \bl v,w \br_F\,  
\end{equation}
with associated induced norms $\|\cdot\|_h^{}$, $\|\cdot\|_{\mathcal{G}}^{}$, respectively.
%As a local version of these definitions, for each $T\in\calTt_h$, we introduce the broken norm \begin{equation}
%	\label{eq:brokennorm}
%	\|v\|_{h,T}^2:=\sum_{\{K\in\calT_h:K\subset T\}}\|v\|^2_{0,K}.
%\end{equation}
We also consider, for all $k$, the semi-norms \[|u|_{k,h}:=\sum_{|\alpha|=k}\|D^\alpha u\|_{h}.\]

Regarding finite element spaces, we define, for $s\ge 1$, the standard piecewise polynomial Lagrange space:
\begin{equation}\label{pol-space-lg}
\bfW_h^s:=\{ \bfw \in H_0^1(\Omega)^d: \bfw|_K^{} \in \pol_s(K)^d \text{ for all }  K \in \calT_h\},
\end{equation}
and for $s\ge 0$, the $DG$ space
\begin{subequations}\label{pol-space-dg}
\begin{alignat}{1}
\bfD^s:=&\{ \bfw \in L^2(\Omega)^3: \bfw|_K^{} \in \pol_s(K)^3 \text{ for all } K \in \calT_h\}, \quad \text{when } d=3, \\
 D_h^s:=&\{ w \in L^2(\Omega): w|_K^{} \in \pol_s(K) \text{ for all } K \in \calT_h\}, \quad \text{when } d=2.
\end{alignat}
\end{subequations}
Over the triangulation $\calT_h^{}$, and for $k\ge 1$, we assume that there exist finite element spaces $\bfV_h \subset H_0^1(\Omega)^d,Q_h \subset L^2_0(\Omega)$, along with the discrete subspace of divergence-free functions arising from the primal weak formulation \eqref{eq:weak:Oseen-reduced} of Oseen's problem \eqref{eq:oseen}:
\begin{equation}\label{div-free-space-h}
\bcalV_h := \{\bfv_h \in \bfV_h: \mbox{ such
  that } \dive \bfv_h= 0 \mbox{ in } \Omega\}\,.
\end{equation}
These spaces must satisfy very similar assumptions to the ones \cite{BarrBurmanGuzman}, (below, and in the rest of the manuscript, $C$ will denote
a positive constant, independent of $h$ and the viscosity $\mu$):
\begin{itemize}
	\item[(A1)] It holds $\dive \bfV_h \subset Q_h$.
	\item[(A2)] The pair $(\bfV_h , Q_h)$ is inf-sup stable.
	\item[(A3)] The inclusions $\bfW_h^k \subset \bfV_h \subset \bfW_h^{r}$ hold for some $r, k \ge 1$.
	\item[(A4)] There exists a finite element space $\bZ_h \subset \bZ$ such that $ \curl \bZ_h=\bcalV_h$.
	\item[(A5)] Let $\mathrm{P}_h^{}$ be the orthogonal projection onto $\bZ_h^{}$ defined by
\begin{equation}\label{def:L2-projection-Zh}
(\bfv-\mathrm{P}_h^{}\bfv,\bfpsi_h^{})_\Omega^{}=0\qquad\forall\, \bfpsi_h^{}\in \bZ_h^{}\,.
\end{equation}
Then, the following error estimate holds:  for all  $\bz \in   \bZ$ with components in  $H^{k+2}(\Omega)$ and  $|\bfalpha|\le k+1$, it holds
	\begin{equation*}
		\|h^{|\bfalpha|} \partial^{\bfalpha} (\bz -\mathrm{P}_h^{}\bz)\|_{h}^{} \le C h^{k+2} \|\bz\|_{k+2,\Omega}^{}.
	\end{equation*}
\end{itemize}
In fact, (A1)-(A4) exactly appear  as assumptions in \cite{BarrBurmanGuzman}. The Assumption (A5) is slightly different than the one in \cite{BarrBurmanGuzman}. Here we use the $L^2$-projection $\mathrm{P}_h^{}$ instead of minimizing over the space.  Additionally, as our goal is to remove bulk stabilization terms and use facet-based stabilization terms instead, the following approximation property involving the spaces $\bZ_h$ and $\bfD^{r-2}$ is going to be needed for the analysis:
\begin{itemize}
	\item[(A6)] For any $\bfw_h \in \bfD^{r-2}$, it holds
	\begin{alignat*}{1}
		\inf_{\bfm_h \in \bZ_h} \| h^{3/2}(\bfw_h-\bfm_h)\|_h \le   C  (\|h^{2} \jump{\bfw_h}\|_{\calF^i}+ \|h^{3} \jump{\nabla \bfw_h}\|_{\calF^i}). 
	\end{alignat*}
\end{itemize}

%%%%%%%%%%%%%%%%%%%%%%%

%%%%%%%%%%%%%%%%%%%%%%%%%%%%%%%%%%%%%%%%%%%%%%%%%%%%%%%%%%%%%%%%%%%%%%%%%%%
\subsection{The discrete method}\label{FEM:method}
The finite element method analyzed in this work reads as follows: find $(\bfu_h^{},p_h^{}) \in \bfV_h^{}\times Q_h^{}$ such that
\begin{equation}\label{eq:FEMstab}
\hspace*{-0.4cm}
\left\{
\begin{array}{rll}
a(\bfu_h^{},\bfv_h^{}) + b(p_h^{},\bfv_h^{})+S(\bfu_h^{},\bfv_h^{}) =&  (\bff,\bfv_h^{})_\Omega \quad & \forall \, \bfv_h \in  \bfV_h\,, \\
b(q_h^{}, \bfu_h^{})=& 0  \quad & \forall \, q_h \in Q_h\,,
\end{array}
\right.
\end{equation}
where the bilinear forms $a$ and $b$ are defined in \eqref{def-a} and \eqref{def-b}, respectively,
and the stabilizing bilinear form $S$ is given by
\begin{equation}
	\label{eq:defS}
	S(\bfu_h,\bfv_h):=\|\bfbeta\|_{0,\infty,\Omega}^{-1} \sum_{j=1}^3\delta_j S_j(\bfu_h,\bfv_h),
\end{equation}
where $S_1$, $S_2$ and $S_3$ are given by
\begin{align}
S_1(\bfu_h,\bfv_h)&:= \bl h^2 \jump{(\bfbeta \cdot \nabla)\bfu_h^{}\times \bfn },\jump{(\bfbeta \cdot \nabla) \bfv_h^{} \times \bfn}\br_{\calF^i} \label{eq:def_s1},\\
S_2(\bfu_h,\bfv_h)&:= \bl h^4 \jump{\calB \bfu_h},\jump{\calB \bfv_h} \br_{\calF^i},\text{ and}\label{eq:def_s2}\\
S_3(\bfu_h,\bfv_h)&:=\bl h^6 \jump{\nabla \calB \bfu_h},\jump{\nabla \calB \bfv_h} \br_{\calF^i}, \label{eq:def_s3}
\end{align}
and the stabilization parameters $\delta_1$, $\delta_2$ and $\delta_3$ are nondimensional and will be set later on. For equations $\eqref{eq:def_s2}$ and $\eqref{eq:def_s3}$ we define $(\calB \bfw)|_T:=\curl \big((\bfbeta \cdot \nabla) \bfw\big)|_T$ for each $T\in\calT_h$.
For the analysis we introduce the following mesh-dependent norm
\begin{equation}\label{norm}
\vertiii{\bfv}^2 := \|\sigma^{\frac12} \bfv\|^2_{0,\Omega} + \|\mu^{\frac12} \nabla\bfv \|^2_{0,\Omega} + |\bfv|_S^2\,,
\end{equation}
where $|\bfv|_S^2 := S(\bfv,\bfv)$. In particular, since $\dive\bfbeta = 0$ we have that
\begin{equation}\label{elipt-a+s}
\vertiii{\bfv_h^{}}^2 =  (a+S)(\bfv_h^{},\bfv_h^{})\qquad \forall\,\bfv_h^{}\in \bfV_h^{}\,.
\end{equation}
In addition, Assumption (A2) ensures the well-posedness of Problem~\eqref{eq:FEMstab}. Moreover,  assuming that  $\bfbeta\cdot\nabla\bfu\in H^{5/2+\varepsilon}(\Omega)^d, \varepsilon>0$, method~\eqref{eq:FEMstab} is strongly consistent for smooth enough $(\bfu,p)$, this is
\begin{equation}\label{eq:gal_ortho}
\hspace*{-0.4cm}
\left\{
\begin{array}{rll}
a(\bfu-\bfu_h^{},\bfv_h^{}) +S(\bfu-\bfu_h^{},\bfv_h^{})+ b(p-p_h^{},\bfv_h^{}) =&  0 \quad & \forall \, \bfv_h \in  \bfV_h\,, \\
b(q_h^{}, \bfu-\bfu_h^{})=& 0  \quad & \forall \, q_h \in Q_h\,. 
\end{array}
\right.
\end{equation}

\section{Error estimates} \label{sec:numerical:analysis}

To avoid technical diversions, in the analysis below we will assume that the family of meshes $\{\mcT_h^{}\}_{h>0}^{}$ is quasi-uniform (nevertheless, all of the results below can be
extended to the more general case by arguing locally).
In the error analysis, the following local trace
inequality will be used repetitively: there exists $C>0$ such that for any $T\in \calT_h^{}$, $F\in\calF_T^{}$, and for any $v\in H^1(T)$, it holds
\begin{equation}\label{local-trace}
\|v\|_{0,F}^{}\le C\Big(h_T^{-\frac{1}{2}}\|v\|_{0,T}^{}+h_T^{\frac{1}{2}}|v|_{1,T}^{}\Big)\,.
\end{equation}
Also, we recall the inverse inequality: for all nonnegative integers $\ell,s,m$ such that $0\le \ell\le s\le m$ and for all $q\in \mathbb{P}_m^{}(T)$ there exists $C>0$ such that
\begin{equation}\label{inverse}
|q|_{s,T}^{}\le Ch_T^{\ell-s}|q|_{\ell,T}^{}\,.
\end{equation}

Finally, as our goal is to obtain error estimates for the velocity with respect to the mesh size $h$ and the viscosity $\mu$, we will not track their dependency on $\bfbeta$, or $\sigma$. Additionally, and for simplicity, we set the stabilization parameters $\delta_j$, $j=1,2,3$ in \eqref{eq:defS} to $1$ throughout this section.

\subsection{An error estimate for the velocity}

In order to capture all the terms involving the bilinear forms $a$ and $S$, we define the following norm for a regular enough function $\bz$:
\begin{equation}\label{fancy-norm}
\| \bz\|_{\star}^2 :=\vertiii{\curl \bz}^2 + (h+\mu)\sum_{s=0}^4h^{2s-4}\|D^s\bz\|_h^2\,.
\end{equation}
Next, we state that the approximation of $\bfu$ with respect to the norm $\vertiii{\cdot}$ given by \eqref{norm} is as good as the approximation of $\bz$ with respect to the norm $\|\cdot\|_\star$ given by \eqref{fancy-norm}, which corresponds to the analogue of \cite[Theorem 6]{BarrBurmanGuzman}.

\begin{theorem}\label{thm:error_global}
Let $\bfu \in  H_0^1(\Omega)^d$ be the solution to \eqref{eq:oseen} and let $\bz$ be its corresponding potential. Let us assume in addition that $\bfbeta$ and $\bfu$ are such that $\bfbeta\cdot\nabla\bfu\in H^{5/2+\varepsilon}(\Omega)^d, \varepsilon>0$, and $\bfbeta|_K^{}\in W^{3,\infty}(K)^d$ for all $K\in\calT_h^{}$.
 Let $(\bfu_h^{},p_h^{})$ be the solution of \eqref{eq:FEMstab}. In addition, we assume  that $\bfbeta\cdot\bfn=0$ on $\partial \Omega$.  Then, the following error estimate holds for a constant $C$ independent of $h$ and $\mu$:
\begin{equation}\label{error-est-1}
\vertiii{\bfu-\bfu_h^{}} \le  C \|\bz-\mathrm{P}_h\bz\|_{\star}
\end{equation}
where $\mathrm{P}_h^{}\bz$ stands for the $L^2$ orthogonal projection onto $\bZ_h^{}$ defined in \eqref{def:L2-projection-Zh}.
\end{theorem}

\begin{proof}
Let $\bfe = \bfu-\bfu_h^{}$, $\bfpsi_h=\mathrm{P}_{h}^{}\bz$ and set $\bfw_h^{}:=\curl \bfpsi_h^{}$. We note that $\bfw_h^{} \in \bcalV_h^{}$ and then, using the Galerkin orthogonality \eqref{eq:gal_ortho} we have 
\begin{equation}\label{the-first}
\vertiii{\bfe}^2= a(\bfe,\bfu-\bfw_h^{})+S(\bfe, \bfu-\bfw_h^{}). 
\end{equation}
We furthemore decompose $\bfe=\bfe_h+ \bfeta_h$ where $\bfe_h= \bfw_h-\bfu_h$ and $\bfeta_h=\bfu-\bfw_h$, and then bound the right-hand side of \eqref{the-first} term by term. For the rest of the proof, $\epsilon>0$ is arbitrary and will be chosen later. Using Cauchy--Schwarz's and Young's inequalities, and the definition of the norms $\vertiii{\cdot}$ and $\|\cdot\|_{\star}$, we see that
\begin{equation}\label{error-proof-3}
S(\bfe, \bfu - \bfw_h^{}) \le \epsilon \vertiii{\bfe}^2 + C \vertiii{\bfu - \bfw_h}^2 \le \epsilon \vertiii{\bfe}^2 + C \|{\bz - \bfpsi_h}\|_{\star}^2.
\end{equation}
Then, it only remains to bound $a(\bfe,\bfu - \bfw_h^{})$. By definition
\begin{equation}\label{error-proof-4}
a(\bfe, \bfu - \bfw_h^{}) =  (\sigma \bfe, \bfu - \bfw_h)_{\Omega} + (\mu \nabla\bfe,\nabla (\bfu - \bfw_h^{}))_\Omega\,+ ((\bfbeta \cdot \nabla) \bfe, \bfu - \bfw_h^{})_\Omega^{}.
\end{equation}
To bound the first and second  term on the right-hand side of \eqref{error-proof-4}, we use the Cauchy-Schwarz inequality and Young's inequality:
\begin{equation*}
(\sigma \bfe, \bfu - \bfw_h^{})_{\Omega}^{} + (\mu \nabla\bfe,\nabla (\bfu - \bfw_h^{}))_\Omega \le  \epsilon\vertiii{\bfe}^2+ C \vertiii{\bfu - \bfw_h^{}}^2.
\end{equation*}
For the last term in \eqref{error-proof-4}, we integrate by parts and obtain
\begin{alignat}{1}
 ((\bfbeta \cdot \nabla) \bfe, \bfu - \bfw_h^{})_\Omega^{}=T_1+T_2, \label{eq:res_stab}
\end{alignat}
where
\[
T_1:=\bl \jump{(\bfbeta \cdot \nabla) \bfe \times \bfn}, \bz-\bfpsi_h \br_{\calF},\quad\text{and}\quad T_2:=(\calB \bfe, \bz-\bfpsi_h)_h^{}.\]
To bound $T_1$,  since $\bfbeta \cdot \bfn = 0$ and $\bfe=0$ on $\partial\Omega$, we have $\bfbeta \cdot\nabla\bfe = \bfzero$ on $\partial\Omega$, and consequently
\[\bl \jump{(\bfbeta \cdot \nabla) \bfe \times \bfn}, \bz-\bfpsi_h^{} \br_{\calF} = \bl \jump{(\bfbeta \cdot \nabla) \bfe \times \bfn}, \bz-\bfpsi_h^{} \br_{\calF^i}^{}\,.\]
Then, we use Young's inequality and the local trace theorem \eqref{local-trace} to get
\begin{equation}
	\label{eq:est1}
	\begin{split}
 T_1= \bl \jump{(\bfbeta \cdot \nabla) \bfe \times \bfn}, \bz-\bfpsi_h \br_{\calF^i} \le & 
\,\epsilon \,\| h \jump{(\bfbeta \cdot \nabla) \bfe \times \bfn}\|_{\calF^i}^2+ C \| h^{-1} (\bz-\bfpsi_h^{})\|_{\calF^i}^2  \\
\le & \, \epsilon\,|\bfe|_S^2 +   C\,h\,\sum_{s=0}^1 h^{2s-4}\|D^s(\bz-\bfpsi_h^{})\|^2_h\\
   \le & \,\epsilon\, \vertiii{\bfe}^2 + C\|\bz-\bfpsi_h^{}\|_{\star}^2.
	\end{split}
\end{equation}
To bound $T_2$, we further decompose $\bfe:=\bfe_h^{}+ \bfeta_h^{}$, where $\bfe_h^{}=\bfw_h^{}-\bfu_h^{}$ and  $\bfeta_h^{}:=\bfu-\bfw_h^{}$. We then write
\begin{equation*}
T_2= T_3+T_4,
\end{equation*}
where 
\[T_3:= (\calB \bfeta_h^{}, \bz-\bfpsi_h^{})_h^{},\quad\text{and}\quad T_4:=(\calB \bfe_h^{}, \bz-\bfpsi_h^{})_h^{}.\]
Then, we have by using Young's inequality and taking the product rule into account that
\begin{equation*}
	\begin{split}
	T_3 & = \big(\curl\big(\big(\bfbeta\cdot\nabla)\bfeta_h^{}\big),\bz-\bfpsi_h^{}\big)_h^{}\\
		& \le h^3\big|(\bfbeta\cdot\nabla)\bfeta_h^{}\big|_{1,h}^2+h^{-3}\|\bz-\bfpsi_h^{}\|_{h}^2\\
		& \le C\|\bfbeta\|_{0,\infty,\Omega}^2h^3|\bfu-\bfw_h^{}|_{2,h}^2+C\|\bfbeta\|_{1,\infty,\Omega}^2h^3|\bfu-\bfw_h^{}|_{1,h}^2+h^{-3}\|\bz-\bfpsi_h^{}\|_{h}^2\\
		& \le C\|\bfbeta\|_{1,\infty,\Omega}^2h^3|\bz-\bfpsi_h^{}|_{3,h}^2+C\|\bfbeta\|_{1,\infty,\Omega}^2h^3|\bz-\bfpsi_h^{}|_{2,h}^2+h^{-3}\|\bz-\bfpsi_h^{}\|_{h}^2\\
		& \le C\|\bz-\bfpsi_h\|_{\star}^2\,.
	\end{split}
\end{equation*}
In order to bound $T_4$, we introduce the approximation $\bfbeta_h^{}$ of $\bfbeta$ as the lowest order Raviart-Thomas interpolant of $\bfbeta$ (see, \cite[Chapter~16]{EG21-I}),
that is, a piecewise constant function that belongs to $H(\textrm{div},\Omega)$, and that satisfies $\textrm{div }\bfbeta_h^{}=0$ in $\Omega$.
%we introduce the local convective term $\bfbeta_h$ which locally belongs to $H(\mathrm{div})$ and to the lowest order Raviart-Thomas space, thus a piecewise constant vector as $\dive\bfbeta = 0$, and
 Then, we define $\calB_h$ for any $T\in\calT_h$ by \[\calB_h \bfw|_T:=\curl \big((\bfbeta_h \cdot \nabla) \bfw|_T\big).\]Then, we decompose
\begin{equation*}
T_4:=T_5+ T_6,
\end{equation*}
where 
\[
T_5:= (\calB \bfe_h- \calB_h \bfe_h, \bz-\bfpsi_h)_h^{}\quad\text{and}\quad
T_6:=(\calB_h \bfe_h, \bz-\bfpsi_h)_h^{}.
\]
To bound $T_5$, we integrate by parts and use that on each element $T\in\calT_h^{}$, the estimate $\|\bfbeta-\bfbeta_h^{}\|_{\infty,T}\le C h_T^{}\|\bfbeta\|_{1,\infty,T}^{}$ holds, local trace and inverse estimates. Also, since we do not track the dependence of error estimates on $\sigma$, we have the estimate
 \begin{equation}\label{est-eta}
\|\bfeta_h^{}\|_h^{}=\|\bfu-\bfw_h^{}\|_h^{}=\|\curl(\bz-\bfpsi_h^{})\|_{h}^{}\le C\vertiii{\curl(\bz-\bfpsi_h^{})}\le C\|\bz-\bfpsi_h^{}\|_{\star}^{}.
\end{equation}
Then, using the above and the estimate for $\|h^{-1}(\bz-\bfpsi_h)\|_{h,\calF^i}$  already derived in \eqref{eq:est1}, we arrive at
\begin{equation*}
	\begin{split}
		T_5 & = \big((\bfbeta-\bfbeta_h^{})\cdot\nabla\bfe_h^{},\curl(\bz-\bfpsi_h^{})\big)_{h}^{}+\bl \jump{(\bfbeta-\bfbeta_h^{})\cdot\nabla\bfe_h^{}\times n},\bz-\bfpsi_h^{} \br_{\calF^i}\\
		& \le Ch\|\bfbeta\|_{1,\infty,\Omega}^{}\|\nabla\bfe_h^{}\|_{h}\|\bz-\bfpsi_h^{}\|_{\star}+Ch^2\|\bfbeta\|_{1,\infty,\Omega}\|\nabla \bfe_h^{}\|_{h,\calF^i}\|h^{-1}(\bz-\bfpsi_h^{})\|_{\calF^i}^{}\\
		& \le C\|\bfe_h^{}\|_{h}^{}\|\bz-\bfpsi_h^{}\|_{\star}^{}+Ch^{1/2}\|\bfe_h^{}\|_{h}^{}\|h^{-1}(\bz-\bfpsi_h^{})\|_{\calF^i}^{}\\
		& \le C\|\bfe_h^{}\|_{h}^{}\|\bz-\bfpsi_h^{}\|_{\star}^{}+Ch^{1/2}\|\bfe_h^{}\|_{h}^{}\|\bz-\bfpsi_h^{}\|_{\star}^{}\\
		& \le C\big(\|\bfe\|_h^{}+\|\bfeta_h^{}\|_h^{}\big)\|\bz-\bfpsi_h^{}\|_\star^{}\\
		& \le \epsilon\vertiii{\bfe}^2+C\|\bz-\bfpsi_h^{}\|_\star^2\,.
	\end{split}
\end{equation*}
To bound $T_6$ we observe that since $\bfpsi_h=\mathrm{P}_h^{}\bz$ and $\mathrm{P}_h^{}$ is the $L^2$ orthogonal projection onto $\bZ_h^{}$, we have for any $\bfm_h^{}\in\bZ_h^{}$ that
\begin{align*}
T_6&=(\mathcal B_h\bfe_h^{},\bz-\bfpsi_h^{})_h^{}\\
&=(\mathcal B_h^{}\bfe_h^{}-\bfm_h^{},\bz-\bfpsi_h^{})_h^{}\\
&\le \big\|h^{3/2}(\mathcal B_h\bfe_h^{}-\bfm_h^{})\big\|_h^{}\big\|h^{-3/2}(\bz-\bfpsi_h^{})\big\|_h^{}\,.
\end{align*}
Therefore, taking the infimum over $\bfm_h^{}\in\bZ_h^{}$, we get
\begin{equation*}
T_6 \le \inf_{\bfm_h^{} \in \bZ_h^{}} \| h^{3/2}(\calB_h \bfe_h^{}-\bfm_h^{})\|_h^{}  \|h^{-3/2} (\bz-\bfpsi_h^{})\|_h^{},
 \end{equation*}
and using that $\calB_h \bfe_h^{} \in \bfD^{r-2}$ and  Assumption (A6), we have that 
\begin{alignat*}{1}
 T_6 \le C (\|h^{2} \jump{\calB_h \bfe_h^{}}\|_{\calF^i}^{}+\|h^{3} \jump{\nabla \calB_h \bfe_h^{}}\|_{\calF^i})^{}\|h^{-3/2} (\bz-\bfpsi_h^{})\|_h^{}. 
\end{alignat*}
In order to bound the jump terms from the above estimate, we add and substract $\calB$ and use local trace inequalities. We will need to bound $|(\calB-\calB_h)\bfe_h^{}|_{\ell,h}^{}$ for $\ell=0,1,2$. To do so, we use the product rule and inverse estimates repeatedly. Since $\bfbeta_h$ is piecewise constant, we have that
\[\|(\calB-\calB_h)\bfe_h^{}\|_{h}^{}\le C\|\bfbeta\|_{1,\infty,\Omega}^{}|\bfe_h^{}|_{1,h}^{}+C\|\bfbeta-\bfbeta_h^{}\|_{0,\infty,\Omega}^{}|\bfe_h^{}|_{2,h}^{}
\le Ch^{-1}\|\bfbeta\|_{1,\infty,\Omega}^{}\|\bfe_h^{}\|_{h}^{}.\]
Similarly, we have that
\begin{equation*}
	\begin{split}
	|(\calB-\calB_h)\bfe_h|_{1,h} & \le C\|\bfbeta\|_{2,\infty,\Omega}|\bfe_h|_{1,h}+C\|\bfbeta\|_{1,\infty,\Omega}|\bfe_h|_{2,h}+C\|\bfbeta-\bfbeta_h\|_{0,\infty,\Omega}|\bfe_h|_{3,h}\\
	&\le Ch^{-2}\|\bfbeta\|_{2,\infty,\Omega}\|\bfe_h\|_{h}.
	\end{split}
\end{equation*}
More generally, we have for $\ell=0,1,2$ that
\begin{equation}
	\label{eq:prodrule}
	|(\calB-\calB_h)\bfe_h^{}|_{\ell,h}\le Ch^{-1-\ell}\|\bfbeta\|_{1+\ell,\infty,\Omega}^{}\|\bfe_h^{}\|_{h}^{}.
\end{equation}
Therefore, by local trace estimates and \eqref{eq:prodrule}, we get
\begin{equation}
	\begin{split}
		&\|h^{2} \jump{\calB_h \bfe_h^{}}\|_{\calF^i}^2+\|h^{3} \jump{\nabla \calB_h \bfe_h^{}}\|_{\calF^i}^2\\
		&\qquad\qquad \le C\|h^{2} \jump{(\calB-\calB_h) \bfe_h^{}}\|_{\calF^i}^2+C\|h^{3} \jump{\nabla (\calB-\calB_h) \bfe_h^{}}\|_{\calF^i}^2+C|\bfe_h^{}|_S^2\\
		& \qquad\qquad \le C\|h^{3/2}(\calB-\calB_h)\bfe_h^{}\|_h^2 + C|h^{5/2}(\calB-\calB_h)\bfe_h^{}|_{1,h}^2\\
		& \qquad\qquad\qquad\qquad + C\|h^{5/2}\nabla(\calB-\calB_h)\bfe_h^{}\|_h^2 + C|h^{7/2}\nabla(\calB-\calB_h)\bfe_h^{}|_{1,h}^2+C|\bfe_h^{}|_S^2\\
		&\qquad\qquad \le Ch\|\bfbeta\|_{3,\infty,\Omega}^2\|\bfe_h^{}\|_{h}^2+C|\bfe_h^{}|_S^2 \\
		&\qquad\qquad \le C\vertiii{\bfe_h^{}}\\
     		&\qquad\qquad \le C \vertiii{\bfe}+C\vertiii{\bfeta_h^{}}\\
		& \qquad\qquad \le C \vertiii{\bfe}+C\|\bz-\bfpsi_h^{}\|_{\star}^{}.
	\end{split}
\end{equation}
Therefore, using Young's inequality, we arrive at
\[T_6\le \epsilon\vertiii{\bfe}^2+C\|\bz-\bfpsi_h\|_{\star}^2.\]
Hence, inserting the bounds for $T_1,\ldots,T_6$  into \eqref{the-first} yields
\begin{alignat*}{1}
\vertiii{\bfe}^2 \le C \epsilon \vertiii{\bfe}^2 +  C  \|\bz-\bfpsi_h\|_{\star}^2.
\end{alignat*}
Taking $\epsilon$ sufficiently small and re-arranging terms finishes the proof. 
\end{proof}

As a consequence of the above result, the following estimate follows from (A5) and (A6) assuming that the solution $\bfu\in H_0^1(\Omega)^d\cap H^{k+1}(\Omega)^d$:
\begin{equation}
	\label{eq:error-est-u}
	\vertiii{\bfu-\bfu_h}\le Ch^{k}\big(h^{\frac12}+\mu^{\frac12}\big)\|\bfu\|_{k+1,\Omega},
\end{equation}
where the constant $C>0$ is independent of $h$ and $\mu$. In the convection-dominated case $\mu\le C\, h$, the above estimate gives, in particular, an $O(h^{k+\frac{1}{2}})$ error
estimate for the $L^2(\Omega)$-norm of the velocity error. In addition, this estimate is pressure robust.

\begin{remark} The need for Assumption~(A6) in the last proof appears in the bound for the term $T_6^{}$. In there, the convective field had to be approximated by a $C^1$ piecewise polynomial function, which prompted the introduction of (A6). In addition, this assumption dictates the exact shape of the extra terms added to the CIP method in order to allow for the proof of the improved error bound in Theorem~\ref{thm:error_global}.
\end{remark}

\subsection{An error estimate for the pressure}
\label{sec:pressure}
For regular enough solutions and data, we prove an error bound for the pressure. This estimate includes, in the convection-dominated regime, a superconvergence result between the discrete pressure and the orthogonal projection of the exact pressure into the finite element space for $p$. As a consequence, an optimal-order error estimate for $p$ arises, with constants independent of the viscosity $\mu$. We denote by 
$\pi_h^{}:L^2(\Omega)\to Q_h^{}$ the orthogonal projector onto $Q_h$, and prove the analogue of \cite[Theorem 8]{BarrBurmanGuzman} for the present stabilization terms.
\begin{theorem}
	\label{eq:p-ph}
	Assume the hypotheses of Theorem~\ref{thm:error_global} hold. Then, there exists
	$C>0$, independent of $h$ and $\mu$, such that
	\begin{equation}
		\label{eq:pip-ph}
		\|\pi_h^{} p-p_h^{}\|_{0,\Omega} \le C\big(1+\mu^{\frac12}+h^{\frac12})\vertiii{\bfu-\bfu_h^{}}.
	\end{equation}	
	As a consequence, the following error estimate follows for the pressure
	\begin{equation}\label{error-p-final}
	\|p-p_h^{}\|_{0,\Omega}^{}\le C\, h^k\,|p|_{k,\Omega}^{}+C\big(1+\mu^{\frac12}+h^{\frac12})\vertiii{\bfu-\bfu_h^{}}\,,
	\end{equation}
	where $C>0$ does not depend on $h$, or $\mu$.
\end{theorem}
\begin{proof}
	Since the pair $(\bfV_h^{},Q_h^{})$ is inf-sup stable due to assumption (A2), we guarantee the existence of $\bfx_h^{}\in\bfV_h^{}$ such that 
\begin{equation}\label{xh}
\dive\bfx_h^{} = \pi_h^{}p - p_h^{}\qquad\textrm{and}\qquad \|\nabla\bfx_h^{}\|_{0,\Omega}\le C\|\pi_h^{} p-p_h^{}\|_{0,\Omega}^{}\,,
\end{equation}
%$\dive\bfx_h = \pi_hp - p_h$, and there exists a constant $C>0$ such that $\|\nabla\bfx_h\|_{0,\Omega}\le C\|\pi_h p-p_h\|_{0,\Omega}$, 
where $C$ depends only on $\Omega$. So, using both statements in \eqref{xh}, Assumption (A1),  the $L^2(\Omega)$-orthogonality property of $\pi_h^{}$, and
 the first equation from the Galerkin orthogonality \eqref{eq:gal_ortho} for $\bfv_h^{} = \bfx_h^{}$,  we get 
\begin{align*}
\|\pi_h^{}p-p_h^{}\|_{0,\Omega}^2 &= (\pi_h^{}p-p_h^{},\dive\bfx_h^{})_{\Omega}^{}\\
&= (p-p_h^{},\dive\bfx_h^{})_{\Omega}^{}\\
&=a(\bfu-\bfu_h^{},\bfx_h^{})+S(\bfu-\bfu_h^{},\bfx_h^{})\\
&\le C \vertiii{\bfu-\bfu_h^{}}\,\vertiii{\bfx_h^{}}-\big((\bfbeta\cdot\nabla)\bfx_h^{},\bfu-\bfu_h^{}\big)_\Omega^{}\,.
\end{align*}

	Therefore, according to \eqref{eq:pip-ph}, it only remains to estimate $\vertiii{\bfx_h^{}}$ in terms of the error $\|\pi_h^{}p-p_h^{}\|_{0,\Omega}^{}$. By the definition of the norm $\vertiii{\cdot}$ and the Poincar\'e inequality, and the fact that we dependency of $\sigma$ on the constant $C>0$ is not being tracked, we have that
	\begin{equation}\label{eq:p-phe1}C\vertiii{\bfx_h^{}} \le \big(1+\mu^{\frac12}\big)\|\pi_h^{} p-p_h^{}\|_{0,\Omega}^{}+|\bfx_h^{}|_S^{},
	\end{equation}
	so that we only have to estimate $|\bfx_h^{}|_S$ in terms of $\|\pi_h^{} p-p_h^{}\|_{0,\Omega}^{}$. Recalling the definition of $S_j$, $j=1,2,3$ given by $\eqref{eq:def_s1}$--$\eqref{eq:def_s3}$, we have by the trace and inverse inequalities that
	\begin{align}
			S_1(\bfx_h,\bfx_h) & \le C\sum_{T\in\calT_h}\big(h_T\|(\bfbeta\cdot \nabla)\bfx_h\|_{0,T}^2+ h_T^3|(\bfbeta\cdot \nabla)\bfx_h|_{1,T}^2\big)\nonumber\\
			& \le C\sum_{T\in\calT_h}\big(h_T\|\bfbeta\|_{0,\infty,T}\|\nabla\bfx_h\|_{0,T}^2+ h_T(1+h_T^2)\|\bfbeta\|_{1,\infty,T}^2\|\nabla\bfx_h\|_{0,T}^2\big)\nonumber\\
			& \le C h \|\nabla\bfx_h\|_h^2\le C h \|\pi_h p-p_h\|_{0,\Omega}^2\,. 		\label{eq:p-phS1}
	\end{align}
%	where $C$ depends on $\|\bfbeta\|_{1,\infty,\Omega}$ and constants arising from the trace and inverse inequalities.	
%This shows that \[S_1(\bfx_h,\bfx_h)^{\frac12}\le Ch^{\frac12}\|\pi_h p-p_h\|.\]
	For $S_2$, we similarly have that
	\begin{align} 
			S_2(\bfx_h^{},\bfx_h^{}) & \le C\sum_{T\in\calT_h^{}}\big(h_T^3|(\bfbeta\cdot \nabla)\bfx_h^{}|_{1,T}^2+ h_T^5|(\bfbeta\cdot \nabla)\bfx_h^{}|_{2,T}^2\big) \nonumber\\
			& \le C\sum_{T\in\calT_h^{}}\|\bfbeta\|_{2,\infty,T}^2\big(h_T^{}\|\nabla\bfx_h^{}\big\|_{0,T}^2+ h_T^3\|\nabla\bfx_h^{}\|_{1,T}^2+h_T^5\|\nabla\bfx_h^{}\|_{2,T}^2\big) \nonumber\\
			& \le C\|\bfbeta\|_{2,\infty,\Omega}\sum_{T\in\calT_h^{}}(h_T+h_T^3+h_T^5)\|\nabla\bfx_h^{}\|_{0,T}^2 \nonumber\\
			& \le C h \|\nabla\bfx_h^{}\|_h^2\le C h \|\pi_h^{} p-p_h^{}\|_{0,\Omega}^2\,.\label{eq:p-phS2}
		\end{align}
%	where, once again, $C$ depends on $\|\bfbeta\|_{2,\infty,\Omega}$ and constants arising from the trace and inverse inequalities. This shows that \[S_2(\bfx_h,\bfx_h)^{\frac12}\le Ch^{\frac12}\|\pi_h p-p_h\|.\]
	Finally, and proceeding similarly as in \eqref{eq:p-phS2},
	\begin{align}
			S_3(\bfx_h^{},\bfx_h^{}) & \le C\sum_{T\in\calT_h^{}}\big(h_T^5|(\bfbeta\cdot \nabla)\bfx_h^{}|_{2,T}^2+ h_T^7|(\bfbeta\cdot \nabla)\bfx_h^{}|_{3,T}^2\big)\nonumber\\
			& \le C\|\bfbeta\|_{3,\infty,\Omega}\sum_{T\in\calT_h^{}}(h_T+h_T^3+h_T^5+h_T^7)\|\nabla\bfx_h^{}\|_{0,T}^2 \nonumber\\
			& \le C h \|\nabla\bfx_h^{}\|_h^2\le C h \|\pi_h^{} p-p_h^{}\|_{0,\Omega}^2\,. \label{eq:p-phS3}
		\end{align}
	Putting estimates $\eqref{eq:p-phS1}$--$\eqref{eq:p-phS3}$ together, we obtain \[|\bfx_h^{}|_S\le Ch^{\frac12}\|\pi_h^{}p-p_h^{}\|_{0,\Omega}^{},\]and combining this with \eqref{eq:p-phe1}, we have the desired result.
\end{proof}

\begin{remark}
The analysis carried out in the last section hinges on the regularity of the convective term $\bfbeta\cdot\nabla\bfu$. Although this is a restriction for certain flows, it is worth making the following remark about this requirement. If we think about the Oseen problem as a simplified model for the Navier-Stokes equation, then it is expected that the convective field $\bfbeta$ has the same regularity as the exact velocity $\bfu$. In addition, in order to prove optimal error estimates (not only the ones proven in this paper, but even standard error estimates), the velocity is required to belong to, at least, $H^{k+1}(\Omega)^d$ whenever discrete velocities of degree $k$ are used. Hence, for polynomial degrees $k\ge 4$ the exact velocity is assumed to belong to, at least, $H^5(\Omega)^d$, and in such a case the convective term is regular enough so the strong consistency \eqref{eq:gal_ortho} holds. So, the regularity required in the convective term is, in theory, only a restriction for the cases $k=2,3$. 
%\textcolor{red}{(do you think that $\bfu\in H^4$ would suffice? if so, $k=3$ also works)}. 
We want to stress that the hypotheses imposed in this work are with the aim of proving the enhanced error estimates from Theorems~\ref{thm:error_global} and \ref{eq:p-ph}, and not to prove standard estimates, or even convergence of the method.
\end{remark}

\section{Verifying Assumptions (A1)--(A6) for the lowest order Scott--Vogelius elements on Clough-Tocher triangulations}
\label{sec:CTFEM}
In this section we will verify (A1)-(A6), for the Scott--Vogelius finite element spaces on Clough-Tocher triangulations. To obtain a Clough-Tocher triangulation,  we start with a shape-regular family of meshes $\big\{\calTt_h\big\}_{h>0}$. Then, we refine each mesh by adding the barycenter  of each triangle to the set of vertices and connecting it to the vertices of that triangle.  We denote the resulting mesh by $\big\{\calT_h\big\}_{h>0}$. The elements of $\calTt_h$ will be named macro-elements and will be denoted by $K$, while the elements of $\calT_h$ will be denoted by $T$. The meshes $\calT_h$ are used in the analysis above and numerical experiments in Section~\ref{sec:numerics}.

\begin{figure}[!h]
	\centering
	\begin{subfigure}[h]{0.45\textwidth}
		\centering
		\includegraphics[width=\textwidth]{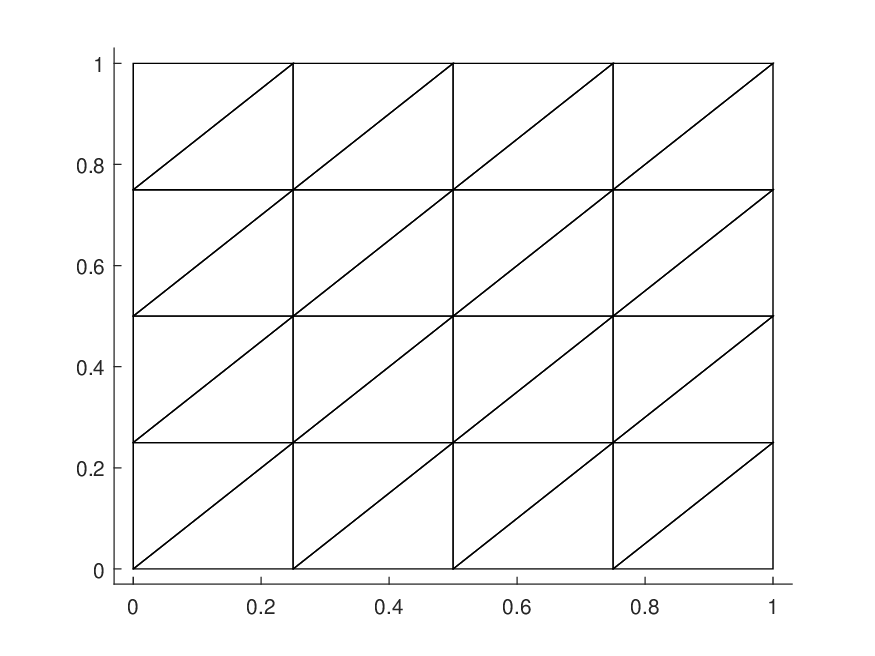}
	\end{subfigure}
	\hfill
	\begin{subfigure}[h]{0.45\textwidth}
		\centering
		\includegraphics[width=\textwidth]{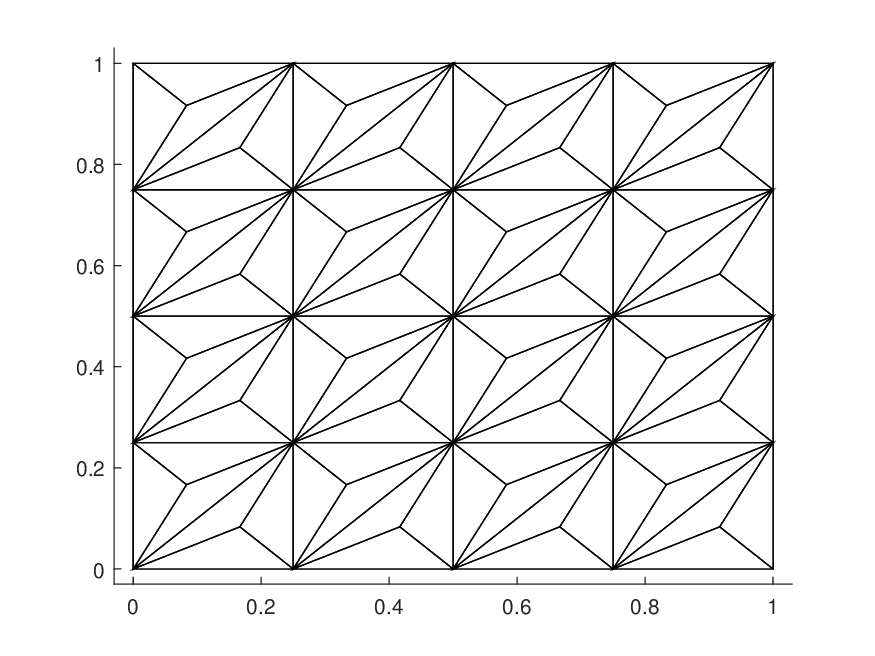}
	\end{subfigure}
	\caption{Initial mesh $\calTt_h$ (left) and its barycentric refinement $\calT_h$ (right).}
\end{figure}
%The Scott-Vogelius finite element space satisfies assumptions (A1)--(A3) and (A5)--(A6) over a %barycentrically refined mesh $\calT_h$. In particular, the lowest order Clough--Tocher $C^1$ %space defined on $\calT_h$ satisfies (A4). 
Let us first define the velocity and pressure spaces for  $k \ge 2$: 
\begin{alignat*}{1}
\bfV_h^{k}:=&\bfW_h^k= \{\bfv \in H_0^1(\Omega)^2: \bfv|_T \in \mathbb P_k(T)^2, \forall T \in \calT_h\} \\
Q_h^{k-1}:=& \{q \in L_0^2(\Omega): q|_T \in \mathbb P_{k-1}(T), \forall T \in \calT_h\}.
\end{alignat*}
The potential space is given by
\[\bZ_h^{k+1} = \big\{\bz_h\in C^1(\Omega):\bz_h|_K\in Z_{k+1}(K)\text{ for all }K\in\calTt_h\big\}, \]
where
\begin{equation*}
	\begin{split}
		D_m(K)&:=\{p\in L^2(K):p|_T\in\mathbb P_m(T)\text{ for all }T\in\calT_h\text{ such that }T\subset K\},\text{ and}\\
		Z_m(K)&:=D_m(K)\cap C^1(K).
	\end{split}
\end{equation*}
We argue here that the triple $\bZ_h=\bZ_h^{k+1}, \bfV_h= \bfV_h^{k}, Q_h=Q_h^{k-1}$ satisfies Assumptions (A1)-(A6). 
The Assumptions (A1)-(A4) were already discussed for these spaces in \cite{BarrBurmanGuzman}. In particular, (A4) follows form the exact sequence properties of spaces on Clough-Tocher splits; see for example, \cite[p. 1068]{FuGuzmanNeilan2018} for a general result in any spatial dimension on  Alfeld splits, which coincides with the Clough-Tocher split in two dimensions.  Assumption (A5) can be proved if we assume for instance that $\calTt_h$ is quasi-uniform. Indeed, in this case we can use inverse estimates, the stability of the $L^2$ projection, and approximation properties of $C^1$ Clough-Tocher finite elements given in \cite[Theorem 2]{dupont1979family}.

The remaining Assumption (A6) follows from a slightly more general result. 
\begin{lemma}\label{A6}
 Let $m \ge 3$. There exists $C>0$, independent of $h$, such that for all $\bfw_h\in \bfD^{m}$, it holds
	\[ \inf_{\bfm_h \in \bZ_h^m} \| h^{3/2}(\bfw_h-\bfm_h)\|_h \le   C  (\|h^{2} \jump{\bfw_h}\|_{\calF^i}+ \|h^{3} \jump{\nabla \bfw_h}\|_{\calF^i}). \]
\end{lemma}

We will prove this result for the particular case $m=3$. The extension to the  general case $m \ge 3$ can easily be extended, however, the main idea of the proof is captured in the case $m=3$. This will allow us to avoid more complicated notation and technical diversions. Our result is a slight generalization to a result found in \cite{emmanuil}. The only difference is that  $\bfw_h|_K \in \mathbb{P}_m(K)$  in \cite{emmanuil} instead of $D_m(K)$. Naturally, our proof will follow closely the corresponding proof in \cite{emmanuil}.

We recall that $\calF$ denotes the set of edges of the mesh  $\calT_h$.   For a macro--element $K\in\calTt_h$, we let
\begin{equation*}
	\begin{split}
		\calE_h^i(K)&:=\{e\in\calF:e\subset K\text{ and }e\not\subset \partial K\}.
	\end{split}
\end{equation*}
We let $\calV_h(\calTt_h)$ denote all the vertices of the mesh $\calTt_h$ and $\calM_h(\calTt_h)$ all the midpoints of all the edges of $\calTt_h$.  Furthermore, we denote  by $\calV_h(K)$ the three vertices of $K$, and by $\calM_h(K)$ the three mid-points of the three edges of $K$. 

The degrees of freedom of the space $Z_m(K)$ for $m \ge 3$ are given in classical references, for example \cite[p. 229]{dupont1979family}. In particular, a function $\bz_h \in Z_3(K)$ is uniquely determined by the following degrees of freedom:
\begin{itemize}
 	\item[a)] $\bz_h(\bfx)$ for $\bfx\in\calV_h(K)$.
 	\item[b)] $\nabla\bz_h(\bfx)$ for $\bfx\in\calV_h(K)$.
 	\item[c)] $\nabla\bz_h(\bfx) \mathbf n$ for $\bfx\in\calM_h(K)$.
 \end{itemize}

We will use the following broken norm $\|\cdot\|_{h,K}$ for $K\in\calTt_h$ in what follows:
 \begin{equation}\label{eq:defhnorm}
 \|\bfv\|_{h,K}^2:=\sum_{\{T\in\calT_h:T\subset K\}}\|\bfv\|_{0,T}^2.
 \end{equation}
 
 A result regarding equivalence on $D_3(K)$ between the norm \eqref{eq:defhnorm} and a norm involving jumps over interior facets and the degrees of freedom arising from $Z_3(K)$ will be needed for the subsequent analysis. Given $\bfx\in K$ and $\bfw_h\in D_3(K)$, we set $\hat\bfx = \bfx / h_K$ and $\hat\bfw_h(\hat\bfx) = \bfw_h(\bfx)$. Then, we observe that the functional
	\begin{equation*}
		\begin{split}
			\|\hat\bfw_h\|_{D_3(\hat K)}^2:=&\|\jump{\hat\bfw_h}\|_{\calE_h^i(\hat K)}^2+\|\jump{\nabla\hat\bfw_h}\|_{\calE_h^i(\hat K)}^2\\
			&+\sum_{\hat\bfx\in\calV_h(\hat K)}\sum_{\{\hat{T}\subset \hat K:\,\hat{\bfx}\in \hat{T}\}}\big|\hat\bfw_h|_{\hat T}(\hat\bfx)\big|^2+\big|\nabla\hat\bfw_h|_{\hat T}(\hat\bfx)\big|^2\\
			&+\sum_{\hat\bfx \in\calM_h(\hat K)}\big|\nabla \hat\bfw_h|_{\hat K}(\hat \bfx)\bfn\big|^2,
		\end{split}
	\end{equation*}
is a norm over $D_3(\hat K)$. Indeed, if   $\jump{\hat\bfw_h}$ and $\jump{\nabla\hat\bfw_h}$ vanish on all the edges of $\calE_h^i(\hat K)$ then $\hat \bfw_h \in Z_3(\hat K)$. Moreover, if the other terms appearing in the definition of $\|\hat\bfw_h\|_{D_3(\hat K)}$ vanish then given the degrees of freedom of $Z_3(\hat K)$,  $\hat\bfw_h$ vanishes. 

By scaling we have 
\[\|\bfw_h\|_{h,K}^2=h_K^2\|\hat\bfw_h\|_{h,\hat K}^2,\]
 and by the equivalence of norms on the finite dimensional space $D_3(\hat K)$ we have
 \[\|\hat\bfw_h\|_{h,\hat K}^2\le C_{\hat K} \|\hat\bfw_h\|_{D_3(\hat K)}^2.\]
 We assume that the constants $C_{\hat K}$ are uniformly bounded over all $K \in \calTt_h$ as is commonly done; see for example \cite[p. 239]{dupont1979family}.
 We therefore have the following estimate for all $\bfw_h \in Z_3(K)$:
	\begin{equation*}
		\begin{split}
					\|\bfw_h\|_{h,K}^2\le & C( \|h_K^{1/2}\jump{\bfw_h}\|_{\calE_h^i(K)}^2+\|h_K^{3/2}\jump{\nabla\bfw_h}\|_{\calE_h^i(K)}^2)\\
					 &+ C \Big(\sum_{\bfx\in\calV_h(K)}\sum_{\{T\subset K:\bfx\in T\}}h_K^2\big|\bfw_h|_T(\bfx)\big|^2+h_K^4\big|\nabla\bfw_h|_T(\bfx)\big|^2 \Big)\\
					&+C \sum_{\bfx\in\calM_h(K)}h_T^4\big|\nabla\bfw_h|_K(\bfx)\bfn\big|^2.
		\end{split}
	\end{equation*}
We can now prove Lemma \ref{A6} in the case of $m=3$.
\begin{proof} (Lemma \ref{A6}, $m=3$)
	%For a function $\bfw_h\in \bfD^{3}$ and a macro-element $K\in\calTt_h$, we let 
	For $\bfx\in\calV_h(\calTt_h)\cup\calM_h(\calTt_h)$, we define
	\[ \omega_{\bfx}:=\{T\in\calT_h:\bfx \in T\},\]
	with $|\omega_{\bfx}|$ denoting the cardinality of $\omega_{\bfx}$. We then define $E:\bfD^{3}\to \bZ_h^3$. If  $\bfx\in\calV_h(\calTt_h)$,
	\[E(\bfw_h)(\bfx)=\dfrac{1}{|\omega_{\bfx}|}\displaystyle\sum_{T\in\omega_{\bfx}}\big(\bfw_h|_T\big)(\bfx),\quad\text{and}\quad\nabla E(\bfw_h)(\bfx)=\dfrac{1}{|\omega_{\bfx}|}\displaystyle\sum_{T\in\omega_{\bfx}}\nabla \big(\bfw_h|_T\big)(\bfx),\]and if $\bfx\in\calM_h(\calTt_h)$, we define
	\[ \big(\nabla E(\bfw_h)\bfn\big)(\bfx)=\dfrac{1}{|\omega_{\bfx}|}\displaystyle\sum_{T\in\omega_{\bfx}}\big(\nabla(\bfw_h|_T)\bfn\big)(\bfx).\]
	Then, we see that $E$ is well defined over $\bfD^{3}$. As it was already mentioned, the construction of the recovery operator $E$ is very similar to the one from \cite[p. 7]{emmanuil}. The main difference lies on the fact that $w_h^{}\in D_3^{}(K)$ is discontinuous across common edges from triangles $T\subset K$, unlike the construction from \cite{emmanuil}, where $w_h^{}\in \mathbb P_3^{}(K)$ is continuous across common edges. By the equivalence of norms over $D_3(K)$ and since $\jump{E(\bfw_h^{})}=\mathbf{0}$ for all  and $\jump{\nabla E(\bfw_h^{})}=\mathbf{0}$ for all interior edges, we have that
	\begin{equation*}
		\begin{split}
			&\|\bfw_h-E(\bfw_h^{})\|_{h,K}^2 \\
			&\le \|h_K^{1/2}\jump{\bfw_h^{}}\|_{\calE_h^i(K)}^2+\|h_K^{3/2}\jump{\nabla\bfw_h^{}}\|_{\calE_h^i(K)}^2\\
			& \quad + \sum_{\bfx\in\calV_h^{}(T)}\sum_{\{T\subset K:\bfx\in T\}}h_K^2\big|\bfw_h^{}|_T(\bfx)-E(\bfw_h^{})(\bfx)\big|^2+h_K^4\big|\nabla\bfw_h^{}|_T(\bfx)-\nabla E(\bfw_h^{})(\bfx)\big|^2\\
			&\quad +\sum_{\bfx\in\calM_h^{}(T)}\sum_{\{T\subset K:\bfx\in T\}}h_K^4\big|\nabla\bfw_h^{}|_T^{}(\bfx)\bfn-\nabla E(\bfw_h^{})(\bfx)\bfn\big|^2.
		\end{split}
	\end{equation*}
We consider a local numbering $T_1,\ldots,T_{|\omega_{\bfx}|}$ of the elements in $\omega_\nu$, so that each consecutive pair $T_j,T_{j+1}$ shares an edge. Then, we recall the arithmetic-geometric mean inequality; if $a_1,\ldots,a_n$ are real numbers, we have for their average $\bar a=(a_1+\ldots+a_n)/n$ that
\[\sum_{j=1}^n(a_j-\bar a)^2\le C_n\sum_{j=1}^{n-1} (a_{j+1}-a_j)^2,\]
for a constant $C_n>0$ that depends on $n$ only. Using the above inequality, the constant $C_n$ depends upon $|\omega_{\bfx}|$ only, that is, on the shape regularity of the mesh. Hence, we have that
	\begin{equation*}
		\begin{split}
			&\sum_{\bfx\in\calV_h(K)}\sum_{\{T\subset K:\bfx\in T\}}h_K^2\big|\bfw_h|_T(\bfx)-E(\bfw_h)(\bfx)\big|^2 \\
			&\qquad =\sum_{\bfx\in\calV_h(K)}\sum_{\{T\subset K:\bfx\in T\}}h_K^2\bigg|\bfw_h|_T(\bfx)-\dfrac{1}{|\omega_{\bfx}|}\sum_{T\in\omega_{\bfx}}\bfw_h|_{T}(\bfx)\bigg|^2\\
			&\qquad \le C \sum_{\bfx\in\calV_h(K)}h_K^{2}\sum_{j=1}^{|\omega_{\bfx}|-1}\big|\bfw_h|_{T_{j+1}}(\bfx)-\bfw_h|_{T_{j}}(\bfx)\big|^2\\
			& \qquad \le C\sum_{e\in \mathcal{G}_h(K)}\|h\jump{\bfw_h}\|_{0,\infty,e}^2\le C\|h^{1/2}\jump{\bfw_h}\|_{\mathcal{G}_h(K)}^2,
		\end{split}
	\end{equation*}
where $\mathcal{G}_h(K)$ is the collection of edges belonging to $\calF^i$ that touch $K$. For the remaining terms involving partial and normal derivative evaluations, we proceed similarly as the last bound and as in the proof of \cite[Lemma~3.1]{emmanuil}, so we omit the details. The proof is then complete by summing over all the macro-elements $K\in\calTt_h$ and putting the above estimates together.
\end{proof}

%We work on the lowest order Clough Tocher finite element.  In this case the dofs are the function value and first derivatives on the macro vertices. The average normal derivative at the midpoint of each macro edge are the final dofs. 

\color{black}

%%%%%%%%%%%%%%%%%%%%%%%%%%%%%%%%%%%%%%%%%%%%%%%

\section{Numerical results}
\label{sec:numerics}

In this section we report the results of two test cases showcasing the performance of the current method. In both cases we consider the domain $\Omega=(0,1)^2$,
and the meshes have been built as follows: we divide the unit square into two triangles starting by tracing the diagonal going from $(0,0)$ to $(1,1)$. Then, we perform a
sequence of red refinements, to obtain a sequence of shape-regular grids. We consider structured meshes, where the vertices of triangles of these are placed uniformly on the unit square, and non-structured meshes, where the vertices are randomly perturbed $\approx 0.07h$ units away by the use of a uniform distribution.
The level 4 non-structured grid considered in our experiments is shown in Fig.~\ref{Fig1}. The corresponding degrees of freedom for the velocity and the pressure are listed next to it. In the tables below, we use the following notations for the velocity errors:
\[L^2(\bfu):=\|\bfu-\bfu_h^{}\|_{0,\Omega}^{}\quad,\quad e_s(\bfu):=\vertiii{\bfu-\bfu_h^{}},\]
and for the pressure errors:
\[L^2(p):=\|p-p_h^{}\|_{0,\Omega}\quad,\quad L^2(\pi_h^{} p):=\|\pi_h^{} p-p_h^{}\|_{0,\Omega}^{}\,,\]
where $\pi_h^{} p$ is the $L^2(\Omega)$-projection of $p$ onto $Q_h$.
\begin{figure}[!h]
		\centering
		\begin{subfigure}[h]{0.45\textwidth}
			\centering
			\includegraphics[width=0.7\textwidth]{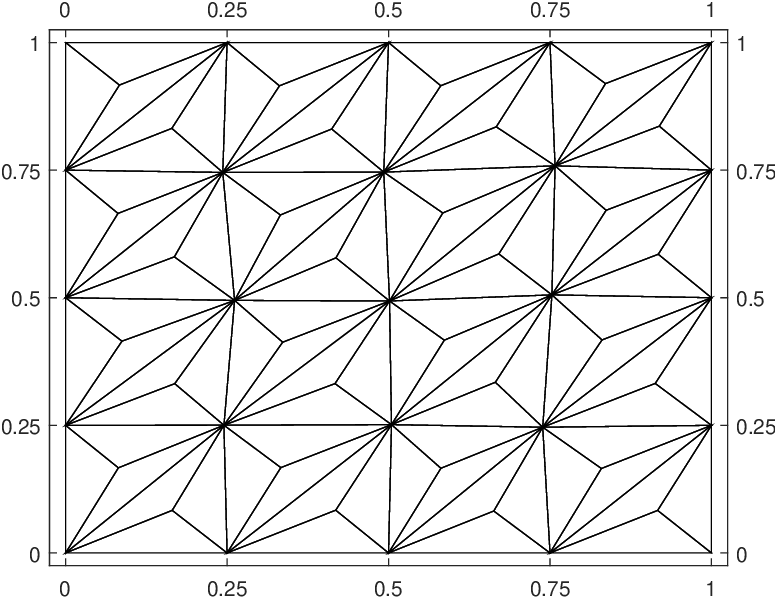}
		\end{subfigure}
		\hfill
		\begin{subfigure}[h]{0.5\textwidth}
				\begin{tabular}{l|lll}
					level & dofs $\bfu_h$ & dofs $p_h$ & total dofs \\
					\hline
					1 & 354 & 288 & 642\\
					2 & 1474 & 1152 & 2626\\
					3 & 6018 & 4608 & 10626 \\
					4 & 24322 & 18432 & 42754 \\
					5 & 97794 & 73728 & 171522 \\
				\end{tabular}
		\end{subfigure}
		\caption{{\it Level 4 non-structured mesh (left) and number of degrees of freedom for all refinement levels (right)}}\label{Fig1}
\end{figure}

\subsection{A problem with a regular solution.} This example combines a planar lattice flow with one whose convective force is constant. We fix $\mu = 10^{-9}$, and set the convective term to $\bfbeta = \bfu + (0,1)^T$, with chosen boundary conditions such that
\[\bfu = \big(\sin(2\pi x)\sin(2\pi y), \cos(2\pi x)\cos(2\pi y)\big)^T,\qquad p = \frac{1}{4}\big(\cos(4\pi x)-\cos(4\pi y)\big).\]
Note that $\bfbeta \cdot \bfn$ does not vanish on the boundary as our theory required. Nonetheless, the expected results are observed numerically as we shall see. Then, $(\bfu,p)$ is the solution to the Oseen problem \eqref{eq:oseen} with the forcing term $\bff$ given by \[\bff = \sigma\bfu-\mu\Delta\bfu+2\pi\big(\sin (2\pi x)\cos (2\pi y),-\cos (2\pi x)\sin (2 \pi y)\big)^T.\] The stabilization parameters $\delta_j$ associated to the bilinear forms $S_j$, $j=1,2,3$ given by $\eqref{eq:def_s1}$, $\eqref{eq:def_s2}$ and $\eqref{eq:def_s3}$ are given by 

\[\delta_1 = 1e-2,\quad \delta_2=1e-5,\quad \delta_3=1e-4.\]

In Table~\ref{tab:t1} we report the errors obtained for two values of $\sigma$ on non-structured or distorted meshes, and in Table~\ref{tab:t2} we do the same on uniform meshes instead. We can observe that the $O(h^{2.5})$ rate for the error $L^2(\bfu)$ is reached early in the calculations, and we can also observe that the pressure converges with the orders predicted by Theorem~\ref{eq:p-ph}, including the super-convergence for $p_h^{}-\pi_h^{}(p)$. In addition, numerical results (not shown here) obtained using only the first stabilising term (that is, using $\delta_2^{}=\delta_3^{}=0$) show an order of convergence $r_0^{}(\bfu)=2.0$, further stressing the need to add the extra terms to the CIP method in order to obtain the enhanced order of convergence given in Theorem~\ref{thm:error_global}. The rate $r_1^{}(\bfu)$ is slightly below what the theory would make us expect, but we believe this is due to a pre-asymptotic regime caused by the very small value of the viscosity $\mu$.

\begin{table}[!ht]
	\begin{tabular}{l|c|ll|ll|ll|ll}
		$\sigma$ & ref & $L^2(\bfu)$ & $r_0(\bfu)$ & $e_s(\bfu)$ & $r_1(\bfu)$ & $L^2(p)$ & $r_0(p)$ & $L^2(\Pi _hp)$ & $r_1(p)$ \\
		\hline
		\multirow{5}{*}{0} & 
		1 & 3.53e-1 & ------ & 1.28e+0  & ------ & 1.11e+0 & ------ & 1.11e+0 & ------\\
		& 2 & 4.55e-2 & 2.981  & 1.70e-1  & 2.937  & 1.17e-1 & 3.276 & 1.12e-1 & 3.337\\
		& 3 & 6.80e-2 & 2.812  & 2.90e-2  & 2.621  & 1.72e-2 & 2.836 & 1.50e-2 & 2.984\\
		& 4 & 1.01e-3 & 2.785  & 5.77e-3  & 2.366  & 3.05e-3 & 2.357 & 2.16e-3 & 2.841\\
		& 5 & 1.41e-4 & 2.870  & 1.12e-3  & 2.379  & 6.20e-4 & 2.320 & 3.07e-4 & 2.838\\
		\hline
		\multirow{5}{*}{1} & 
		1 & 3.01e-1 & ------ &  1.29e+0 & ------ & 1.04e+0 & ------ & 1.04e+0 & ------\\
		& 2 & 3.76e-2 & 3.027 & 1.71e-1 & 2.941  & 1.16e-1 & 3.191 & 1.12e-1 & 3.252\\
		& 3 & 5.34e-3 & 2.889 & 2.92e-2 & 2.618  & 1.73e-2 & 2.819 & 1.51e-2 & 2.964\\
		& 4 & 7.39e-4 & 2.900 & 5.80e-3 & 2.373  & 3.06e-3 & 2.544 & 2.16e-3 & 2.848\\
		& 5 & 9.54e-5 & 2.978 & 1.12e-3 & 2.382  & 6.20e-4 & 2.322 & 3.07e-4 & 2.841\\
	\end{tabular}
	\caption{{\it Error plots and rates of convergence of different norms on different refinement levels on distorted meshes for Scott-Vogelius with the current CIP stabilization ($\sigma = 0$ above and $\sigma = 1$ below) and fixed viscosity $\mu=10^{-9}$.}}
	\label{tab:t1}
\end{table}

\begin{table}[!ht]
	\begin{tabular}{l|c|ll|ll|ll|ll}
		$\sigma$ & ref & $L^2(\bfu)$ & $r_0(\bfu)$ & $e_s(\bfu)$ & $r_1(\bfu)$ & $L^2(p)$ & $r_0(p)$ & $L^2(\Pi _hp)$ & $r_1(p)$ \\
		\hline
		\multirow{5}{*}{0} & 
		1 & 3.38e-1 & ------ & 1.14e+0  & ------ & 9.74e-1 & ------ & 9.77e-1 & ------\\
		& 2 & 4.17e-2 & 3.014  & 1.15e-1  & 2.904  & 9.97e-2 & 3.288 & 9.45e-2 & 3.370\\
		& 3 & 6.57e-2 & 2.671  & 2.65e-2  & 2.532  & 1.50e-2 & 2.730 & 1.23e-2 & 2.933\\
		& 4 & 1.05e-3 & 2.644  & 5.35e-3  & 2.307  & 2.82e-3 & 2.412 & 1.82e-3 & 2.760\\
		& 5 & 1.61e-4 & 2.698  & 1.05e-3  & 2.341  & 5.98e-4 & 2.237 & 2.66e-4 & 2.775\\
		\hline
		\multirow{5}{*}{1} & 
		1 & 2.90e-1 & ------ &  1.15e+0 & ------ & 9.18e-1 & ------& 9.21e-1 & ------\\
		& 2 & 3.48e-2 & 3.060 & 1.54e-1 & 2.903  & 9.95e-2 & 3.206 & 9.43e-2 & 3.287\\
		& 3 & 5.09e-3 & 2.776 & 2.67e-2 & 2.531  & 1.51e-2 & 2.718 & 1.24e-2 & 2.917\\
		& 4 & 7.51e-4 & 2.760 & 5.38e-3 & 2.314  & 2.82e-3 & 2.419 & 1.83e-3 & 2.767\\
		& 5 & 1.05e-5 & 2.829 & 1.05e-3 & 2.346  & 5.98e-4 & 2.239 & 2.66e-4 & 2.779\\
	\end{tabular}
	\caption{{\it Error plots and rates of convergence of different norms on different refinement levels on uniform meshes for Scott-Vogelius with the current CIP stabilization ($\sigma = 0$ above and $\sigma = 1$ below) and fixed viscosity $\mu=10^{-9}$.}}
	\label{tab:t2}
\end{table}

\newpage

\subsection{A solution with a boundary layer.}\label{sec:ex2} In this example, we illustrate why stabilization is needed in order to obtain a good approximation of $\bfu$ and $p$ and treat boundary layers appropriately, as well as showing how many stabilization terms are enough to achieve this. We fix $ \mu = 10^{-8}$  and $\sigma = 0$, and the convective term is given by $\bfbeta = (1,0)^T$. The boundary conditions are chosen so that $(\bfu,p)$ is the solution to the Oseen problem \eqref{eq:oseen} with $\bff = \mathbf{0}$, with
\[\bfu(x,y) = \Bigg(0, x-\frac{\exp\big(\tfrac{x-1}{\mu}\big)-\exp\big(\tfrac{-1}{\mu}\big)}{1-\exp\big(\tfrac{-1}{\mu}\big)}\Bigg)^T,\qquad p(x,y) = \frac12 - y.\]
We observe that $\bfu$ is divergence-free, and that its second component experiences very large gradients near the line $x=1$. We expect that stabilization for controlling this large magnitude derivatives will be needed. In order to show this, we choose three different sets of stabilization parameters $\delta_j$ associated to the bilinear forms $S_j$, $j=1,2,3$ given by \eqref{eq:def_s1}, 
\eqref{eq:def_s2} and \eqref{eq:def_s3}. We take a mesh like in Fig.~\ref{Fig1} with $h\approx 0.0442$, and the comparisons are performed by depicting cross sections
of the second component of the velocity along the line $y=0.5$.
First, we consider $(\delta_1,\delta_2,\delta_3)=(0,0,0)$, that is, no stabilization at all. From Figure~\ref{fig:ex2-dd} (top left), we see that the approximate solution is very oscillatory and there is pollution throughout the domain. Then, as intermediate cases (top right and bottom left), we consider $(\delta_1,\delta_2,\delta_3)=(0.1,0,0)$ and the \textit{classical} CIP method \cite{BFH06} with $\delta=0.1$, respectively. By classical CIP method we mean that the stabilizing term is given by
\[S(\bfu_h^{},\bfv_h^{}):= \delta\,\|\bfbeta\|_{0,\infty,\Omega}^{-1}\,\bl h^2 \jump{(\bfbeta \cdot \nabla)\bfu_h^{}},\jump{(\bfbeta \cdot \nabla) \bfv_h^{} }\br_{\calF^i}^{}.\]
We observe that strong oscillations still persist for these cases. Finally, we take $(\delta_1,\delta_2,\delta_3)=(0.1,0.01,0.001)$ (bottom right), which is the fully stabilized method considered in this work. For this case, although a slight oscillation localized near the boundary layer appears, it is of a smaller magnitude than the one for the CIP method and the partial stabilization, and there is a total absence of  oscillations away from the layer. 

From the results depicted in Figure~\ref{fig:ex2-dd} we conclude by the above experiments that oscillations persist in both cases $(\delta_1,\delta_2,\delta_3) = (0.1,0,0)$ and the classical CIP method defined above. Thus, the need to add stabilization on the higher order derivatives over the facets of the triangulation.

\begin{figure}[!h]
	\centering
 	\begin{subfigure}[h]{0.45\textwidth}
		\centering
		\includegraphics[width=\textwidth]{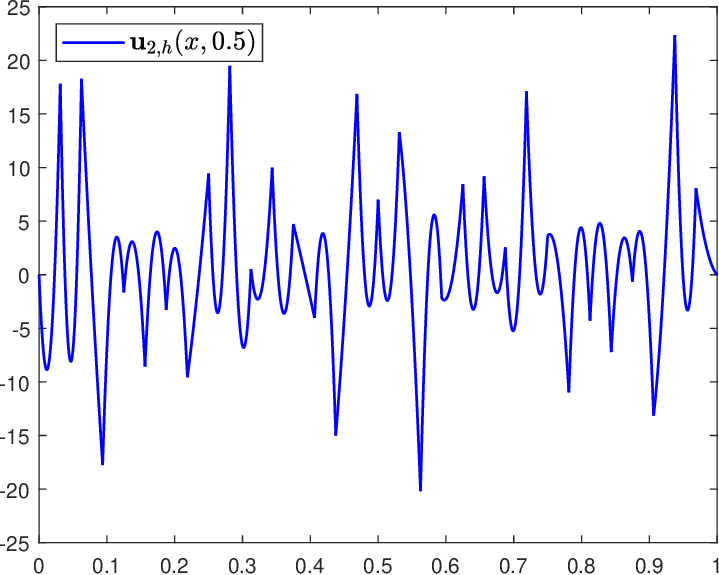}
	\end{subfigure}
 	\hfill
	\begin{subfigure}[h]{0.45\textwidth}
		\centering
		\includegraphics[width=\textwidth]{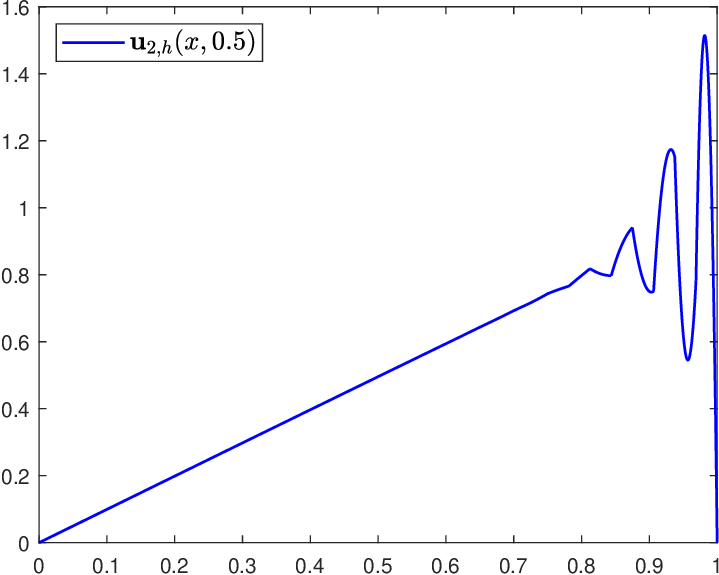}
	\end{subfigure}\\
	\vspace{0.5cm}
	\begin{subfigure}[h]{0.45\textwidth}
		\centering
		\includegraphics[width=\textwidth]{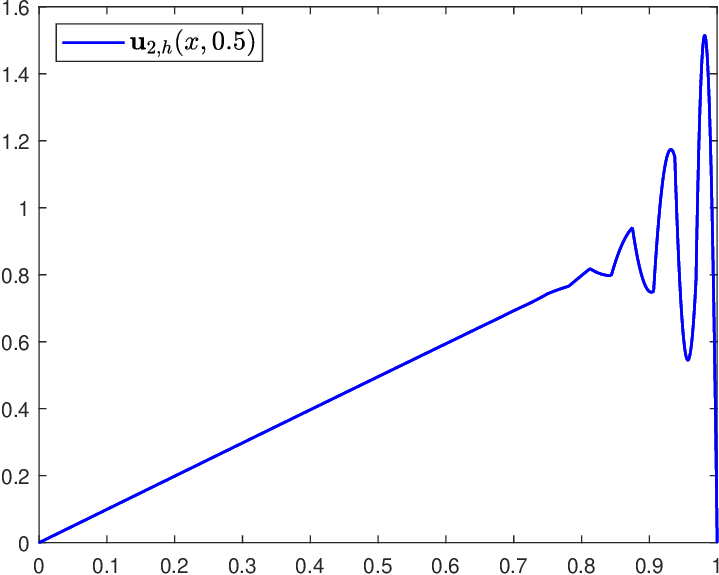}
	\end{subfigure}
	\hfill
	\begin{subfigure}[h]{0.45\textwidth}
		\centering
		\includegraphics[width=\textwidth]{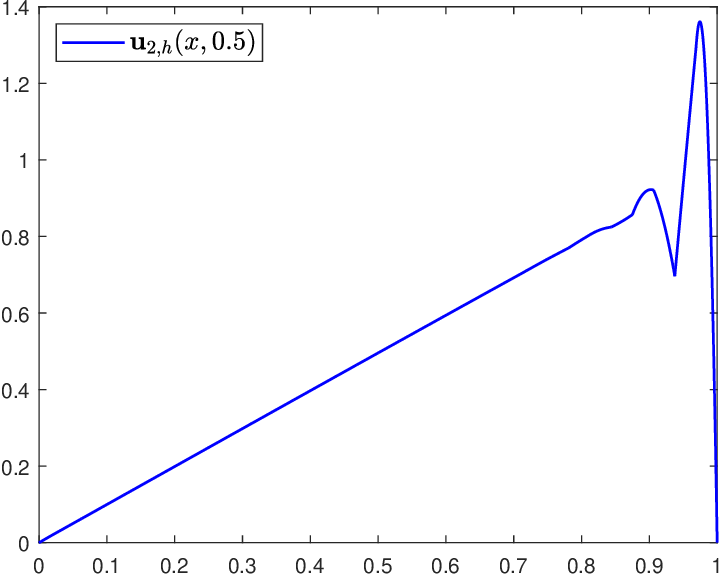}
	\end{subfigure}
	\caption{{\it Example 5.2. Cross sections at $y=0.5$ for the Scott-Vogelius element. No stabilization (top left), partial stabilization (top right) $(\delta_1,\delta_2,\delta_3)=(0.1,0,0)$, Full CIP stabilization (bottom left) and full stabilization (bottom right) $(\delta_1,\delta_2,\delta_3)=(0.1,0.01,0.001)$}.}
	\label{fig:ex2-dd}
\end{figure}

%%%%%%%%%%%%%%%%%%%%%%%%%%%%%%%%%%%%

\section{Conclusions and outlook} \label{concl}
This work extended the applicability of the edge-based stabilization framework to the case where the finite element spaces for the velocity lead to a pointwise divergence-free discrete velocity. More precisely, after realizing that the classical CIP method does not provide enough stability to prove enhanced error estimates for the velocity (remark already hinted in \cite{BL08} where the order of convergence was limited to $k$, rather than the $k+\frac{1}{2}$ provided in this paper), we have proposed to add two extra terms to the edge stabilization. These terms are based on jumps of higher derivatives of the convective term. As such, the method presented in this work is an extension of the one presented in the recent work \cite{tong2022skeletonstabilized}. The main result of the paper is the improved error estimate proven in Theorem~\ref{thm:error_global}, which can only be achieved thanks to the extra terms added to the formulation. This claim is confirmed by our numerical results that show that, while the current method shows the claimed order of convergence, the classical CIP method does not. \\~
Since this work's goal was to show that CIP needs to be enhanced with extra terms to achieve the right stability in the low viscosity limit, there are still several open problems, such as:
\begin{itemize}
    \item the possibility of lowering the regularity required to the convective. While we believe that this is mostly an artifact of the proof, and that it is only an issue appearing in the lower order cases, we acknowledge that this might become important if convergence of the method for less regular solutions is sought.
   \item 
The proof of (A6) shown in Section~\ref{sec:CTFEM} is restricted to the case $d=2$. A similar approach using Alfeld splits can be followed, arguing as in  \cite{FuGuzmanNeilan2018}, to extend this proof to the case $d=3$. 
    \item The extension to time-dependent problems. The design of the method adds a symmetric positive semi-definite term to the equation, which, maybe more importantly than the symmetry itself, does not involve the residual of the strong equation in the bulk. This last point decouples the space stabilization from the time discretization and makes the method extensible to the time-dependent case in a natural way using the techniques presented, e.g., in \cite{BF07}.
\end{itemize}
These topics (among others) are the subject of current research, and will be reported elsewhere.

\section*{Acknowledgements} The work of GRB has been funded by the Leverhulme Trust through the Research Fellowship No. RF-2019-510. EB has been partially funded by the EPSRC grant EP/T033126/1.

\bibliographystyle{abbrv}
\bibliography{references}

\end{document}